\documentclass[12pt,leqno]{article}

\usepackage[lite]{amsrefs}
\usepackage{latexsym,enumerate}

\usepackage[notref, notcite]{}
\usepackage{amssymb, amsfonts, eucal,amsmath,amsthm}
\newtheorem{theorem}{Theorem}[section]

\newtheorem{lemma}[theorem]{Lemma}

\newtheorem{remark}[theorem]{Remark}

\newtheorem{corollary}[theorem]{Corollary}

\newcommand{\sect}[1]{\section{#1} \setcounter{equation}{0} }

\newcounter{ca}
\newcommand{\norm}[1]{\left\|#1\right\|}

\newcommand{\ds}{\displaystyle}

 \newcommand{\ec}{\end{comment}}
\newcommand{\bc}{ \begin{comment}
 }

\newcommand{\andd}{\quad\mbox{\rm and}\quad}

\newcommand\ep{\epsilon}
\newcommand\fep{f_\ep}

\def\be  {\begin{equation}}
\def\ee  {\end{equation}}
\def\ba  {\begin{eqnarray}}
\def\ea  {\end{eqnarray}}
\def\baa {\begin{eqnarray*}}
\def\eaa {\end{eqnarray*}}
\newenvironment{comment}[2]
{\bgroup\vspace{7pt}
\begin{tabular}{|p{5in}|}
\hline \qquad \bf \footnotesize Comment -- to be deleted in the final version \\
\hline
\quad\sl\footnotesize #1#2} {\\ \hline \end{tabular}
\vspace{7pt}\indent\egroup}

\def\updots{\mathinner{\mkern
1mu\raise 1pt \hbox{.}\mkern 2mu \mkern 2mu \raise
4pt\hbox{.}\mkern 1mu \raise 7pt\vbox {\kern 7 pt\hbox{.}}} }

\newcommand{\R}{\mathbb R}

\newcommand{\N}{\mathbb N}

\renewcommand{\a}{\alpha}
\renewcommand{\b}{\beta}
\newcommand{\ineq}[1]{{\rm(\ref{#1})}}

\newcommand{\ie}{{\em i.e., }}
\newcommand{\eg}{{\em e.g., }}

\newcommand{\bpic}{
\begin{center}
}

\newcommand{\epic}{
\endpspicture
\end{center}
}

\newcommand{\st}{\;\; \big| \;\;}

\newcommand{\Poly}{\Pi}
 \newcommand{\AC}{\mathrm{AC}}
  \newcommand{\loc}{\mathrm{loc}}

\newcommand{\thm}[1]{Theorem~$\ref{#1}$}
\newcommand{\lem}[1]{Lemma~$\ref{#1}$}
\newcommand{\cor}[1]{Corollary~$\ref{#1}$}

\newcommand{\rem}[1]{Remark~$\ref{#1}$}

\newcommand{\eps}{\varepsilon}
\newcommand{\geps}{g_\eps}
\newcommand{\gP}{{\mathcal P}}
\newcommand{\Q}{{\mathcal Q}}
\newcommand{\dal}{D^{\langle \alpha \rangle}}
\newcommand{\lupas}{Lupa\c{s}}
\newcommand{\xxx}{\xi}

\title{{\sc Yet another look at positive linear operators, $q$-monotonicity and applications}
\thanks{{\it AMS classification:} 41A10, 41A17, 41A25. {\it Keywords
and phrases:} Positive linear operators, degree
of approximation, Jackson-type estimates, modulus of smoothness, Gavrea's operator, Bernstein-Durrmeyer-Lupa\c{s} polynomials with ultraspherical weights}}

\author{K.  Kopotun\thanks{Department of Mathematics, University of
Manitoba, Winnipeg, Manitoba, R3T 2N2, Canada ({\tt
Kirill.Kopotun@umanitoba.ca}). Supported by NSERC of Canada Discovery Grant RGPIN 04215-15.} ,
D. Leviatan\thanks{Raymond and Beverly Sackler School of Mathematical
Sciences, Tel Aviv University, Tel Aviv 69978, Israel ({\tt
leviatan@post.tau.ac.il}).} , A. Prymak\thanks{Department of Mathematics, University of
Manitoba, Winnipeg, Manitoba, R3T 2N2, Canada ({\tt
prymak@gmail.com}). Supported by NSERC of Canada Discovery Grant RGPIN 04863-15.}\ \ and I. A. Shevchuk\thanks
{Faculty of Mechanics and Mathematics,  Taras
Shevchenko National University of Kyiv, 01033 Kyiv, Ukraine ({\tt
shevchukh@ukr.net}).} }

\begin{document}

%
%
%
%
%
%
%
%



%
%


\maketitle

\abstract{

For each $q\in\N_0$, we construct positive linear polynomial approximation operators $M_n$ that simultaneously preserve $k$-monotonicity for all $0\leq k\leq q$  and yield the estimate
\[
|f(x)-M_n(f, x)|  \leq  c \omega_2^{\varphi^\lambda} \left(f,  n^{-1} \varphi^{1-\lambda/2}(x) \left(\varphi(x) + 1/n \right)^{-\lambda/2}  \right) ,
\]
for $x\in [0,1]$ and  $\lambda\in [0, 2)$,
where $\varphi(x) := \sqrt{x(1-x)}$ and  $\omega_2^{\psi}$ is the second Ditzian-Totik modulus of smoothness corresponding to the ``step-weight function'' $\psi$.
In particular, this implies that the rate of best uniform $q$-monotone polynomial approximation can be estimated in terms of
  $\omega_2^{\varphi} \left(f,  1/n  \right)$.

}

\sect{Introduction and main result}

Recall that
$
{\Delta}_\delta^{k}(f, x):= \sum_{i=0}^{k} { k \choose i }
(-1)^{k-i} f(x-k\delta/2+i\delta),
$
denotes the $k$th symmetric difference of a function $f$ with a
step $\delta$ (as is customary, we also define ${\Delta}_\delta^{k}(f, x):=0$ if $x\pm k\delta/2 \not \in [0,1]$).
We say that a function $f\in C[0,1]$ is   $q$-monotone  if ${\Delta}_\delta^{q}(f, x) \geq 0$  for all $\delta>0$,
and denote the set of all $q$-monotone (continuous) functions by $\Delta^{(q)}$.
In particular,  $\Delta^{(0)}$, $\Delta^{(1)}$ and $\Delta^{(2)}$  are, respectively, the classes of all nonnegative, nondecreasing and convex functions from $C[0,1]$.
We also remark that, for $q\geq 3$, $f\in C[0,1]$ is $q$-monotone if and only if $f\in C^{q-2}(0,1)$ and $f^{(q-2)}$ is convex in $(0,1)$.

Let $\Poly_n$ be the space of all algebraic polynomials of degree $\leq n$, $\norm{\cdot}{} := \norm{\cdot}_{L_\infty[0,1]}$,   and denote
 by
\be \label{qmon}
E^{(q)}_n(f):=\inf_{p_n\in\Delta^{(q)}\cap\Poly_n}\|f-p_n\|
\ee
the degree of best $q$-monotone polynomial approximation of $f\in\Delta^{(q)}$ in the uniform norm, and by
\[
\omega_k (f,t) :=\sup_{0<h\le t}\norm{\Delta^k_{h}(f,\cdot)}  \andd
\omega_k^\psi(f,t) :=\sup_{0<h\le t}\norm{\Delta^k_{h\psi(\cdot)}(f,\cdot)}
\]
the $k$th classical and $k$th Ditzian-Totik  moduli of smoothness, respectively.

Both uniform and pointwise Jackson type estimates for $q$-monotone polynomial approximation are rather well investigated for  $q\leq 3$ though there are still several open problems remaining even in these ``simple'' cases (see our survey \cite{klps} for the history and detailed discussions), and we are mostly interested in   $q\geq 4$ in the current paper. In particular, our main motivation for the present work was the Jackson type estimate
\be\label{jackson}
 E^{(q)}_n(f)\le c\omega_{2}^\varphi(f,1/n), \quad n\in\N,
\ee
where $\varphi(x):= \sqrt{x(1-x)}$ and $\N$ denotes the set of all natural numbers.
It has been known for some time   that estimate \ineq{jackson}  is true with $\omega_2$ instead of $\omega_2^\varphi$ and that, for $q\geq 4$,
it is no longer valid if $\omega_{2}^\varphi$ is replaced by $\omega_{3}^\varphi$  or even by $\omega_{3}$  (see \cite{klps} for details). While \ineq{jackson} has not been explicitly proved anywhere (as far as we know) and appeared as an open problem in the literature (see, \eg \cite{d}*{(15.12)}),
in our survey \cite{klps}*{p. 52}, we wrote that, for $q\ge4$, \ineq{jackson} ``can be derived from results in the article by Gavrea,
Gonska, P{\u a}lt{\u a}nea and Tachev \cite{ggpt},
combined with the $q$-monotonicity preservation properties of the
Gavrea operators (see Gavrea \cite{g}), appearing in the
paper of Cottin, Gavrea, Gonska, Kacs\'o and Zhou \cite{cggkz}.''

However, it turns out that this statement was not justified (we thank Jorge Bustamante from Universidad Aut\'{o}noma de Puebla, Mexico for bringing this to our attention), and that the validity of \ineq{jackson} cannot be immediately concluded from the results in these articles (this was also confirmed by the corresponding author of \cite{cggkz} who was not aware of any other papers that would yield this estimate).  The confusion was that, in these papers, the same notation was used for operators preserving $q$-monotonicity, $q\ge3$, and for  operators   yielding estimates in terms of $\omega_{2}^\varphi(f,1/n)$.
 However, these operators depended on different generating polynomials and so, in fact, were different operators  not satisfying both conditions at the same time.

Hence, the main purpose of this manuscript is to justify/modify our statement in \cite{klps}  and show how \ineq{jackson} ``can be derived from \cites{g,cggkz, felten}'' (note that \cite{ggpt} in our original statement is replaced by an earlier paper \cite{felten}) by
constructing positive linear polynomial approximation operators that simultaneously preserve $k$-monotonicity for all $k\leq q$  and yield \ineq{jackson}. Additionally, we make this paper self-contained and provide all proofs (except for some straightforward statements that can be verified directly and some classical properties of ultraspherical polynomials). Furthermore, we prove a more general statement than \ineq{jackson} by bridging pointwise and uniform estimates (see \cite{d}*{Section 14} for the history of this type of estimates) and, in fact, making them a bit stronger than what usually appears in the literature. For example, {\em pointwise} inequalities in terms of $\omega_2^\varphi$ are obtained as a by-product of our estimates.

Let $\N_0 := \N\cup\{0\}$. Our main result is the following theorem which is proved in Section~\ref{proofmain}.

\begin{theorem} \label{mainth}
 Let $q\in\N_0$. Then, for each $n\in\N$, there exists a positive linear  operator $M_n:C[0,1]\mapsto\Poly_n$ preserving $k$-monotonicity for every $0\leq k \leq q$
 (\ie $f\in\Delta^{(k)}$ implies $M_n(f,\cdot) \in \Delta^{(k)}$)
 and such that, for any $0\leq \lambda<2$, $f\in  C[0,1]$, $x\in [0,1]$ and $0<h\leq c_0$, one has
\be \label{maines}
 |f(x)-M_n(f, x)| \leq c \left(  1 +  {\varphi^{2-\lambda}(x) \over h^2 n^2 \left( \varphi(x) + 1/n \right)^\lambda }  \right) \omega_2^{\varphi^\lambda} (f, h),
 \ee
 where
 $c_0$ is some absolute constant,  and the constant $c$ depends only on $q$ and  on $\lambda$ as $\lambda \to 2-$.
\end{theorem}

\begin{remark}
The operators $M_n$ are particular instances (for the generating polynomials constructed in \lem{lem210})
of, what we call, Gavrea's operators $H_{n}$ whose construction  is based on  Ioan Gavrea's  clever combination of genuine Bernstein-Durrmeyer polynomials  with coefficients of appropriate generating polynomials (see \ineq{gavreaop}). This construction heavily relies on a very powerful but little known and hardly accessible article by Alexandru Lupa\c{s}  \cite{lup}, extending the Bernstein-Durrmeyer operators by introducing ultraspherical weights (see Section \ref{append} for details).
\end{remark}

We wish to emphasize that the range for $\lambda$ in the statement of \thm{mainth} is not a misprint and that, indeed, we work with $\lambda\in [0, 2)$ and not just $\lambda\in [0,1]$ which is what is usually done. This does not seem to have been considered in the literature as far as we know, and we discuss why it is sometimes useful to work with these $\lambda$'s and corresponding moduli $\omega_2^{\varphi^\lambda}$   by considering  an analog of \thm{mainth} for the classical Bernstein polynomials (see \cor{clBern})
 and comparing various estimates for a particular    function ($\fep (x)= x^\ep$) in
Section~\ref{whyuseful}.

We also note that \ineq{maines} is not valid if $\lambda=2$. In fact, it is not difficult to see that the estimate
\[
E_n(f) := \inf_{p_n\in\Poly_n}\|f-p_n\| \leq c \omega_2^{\varphi^2}(f, 1)
\]
is not valid with $c$ independent of $f$. Indeed, if $g_\ep := \ln(x+\ep)$, then
$\omega_2^{\varphi^2} (g_\ep,1) \leq c \norm{\varphi^4 g_\ep''} \leq c$ where $c$ is an absolute constant.
At the same time,
for any $A\in\R$ and $n\in\N$ there exists $0<\ep<1$ such that
$E_n(g_\ep)>A$. This follows from the observations that $|p_n(0)| \leq c(n) \norm{p_n}_{C[1/2,1]}$, for any $p_n \in \Poly_n$, and
 $\norm{g_\ep}_{C[1/2,1]} \leq \ln 2$. Hence, if $q_n\in\Poly_n$ is such that $\norm{q_n - g_\ep} \leq A$, then
 \[
 |\ln \ep| = |g_\ep(0)| \leq |g_\ep(0)-q_n(0)| + |q_n(0)| \leq A + c(n) (A+\ln 2),
 \]
and one obtains a contradiction by taking $\ep>0$ sufficiently small.

For $0\leq \lambda < 2$, choosing  $h:= \min\{c_0, 1\} \, n^{-1} \varphi^{1-\lambda/2}(x) \left(\varphi(x) + 1/n \right)^{-\lambda/2} $ (which implies that $h\leq c_0$)
we immediately have the following consequence of \thm{mainth}.

\begin{corollary} \label{cormainth} Let $q\in\N_0$. Then, for each $n\in\N$, there exists a positive linear  operator $M_n:C[0,1]\mapsto\Poly_n$ preserving $k$-monotonicity for every $0\leq k \leq q$, and such that,
for any $0\leq \lambda <2$, $f\in  C[0,1]$ and $x\in [0,1]$, one has
 \begin{eqnarray} \label{in:coroll}
 |f(x)-M_n(f, x)| &\leq&  c \omega_2^{\varphi^\lambda} \left(f,  n^{-1} \varphi^{1-\lambda/2}(x) \left(\varphi(x) + 1/n \right)^{-\lambda/2}  \right) ,
\end{eqnarray}
where $c$ is a constant that depends only on $q$ and  on $\lambda$ as $\lambda \to 2-$.
In particular, for $\lambda = 0$ and $\lambda=1$ we have, respectively,
 \be \label{in:cg}
 |f(x)-M_n(f, x)| \leq c  \omega_2  \left(f,  \frac{\sqrt{x(1-x)}}{n} \right)
 \ee
and
\be \label{in:dtmod}
|f(x)-M_n(f, x)|
\leq c
 \omega_2^{\varphi} \left(f, n^{-1} \sqrt{ {\varphi(x) \over \varphi(x)+1/n } } \, \right)
\leq c   \omega_2^{\varphi} \left(f,  n^{-1}  \right).
\ee
\end{corollary}

\begin{remark}
Estimate \ineq{in:dtmod} verifies \ineq{jackson}.
Inequality \ineq{in:cg} was proved by Cao and Gonska in $1994$  (\cite{cg}*{Theorem 4.5}). However, the operator yielding it in  \cite{cg}  was not positive.
\end{remark}

\begin{remark} Estimate \ineq{in:coroll} can be rewritten as
\[
|f(x)-M_n(f, x)| \leq c \omega_2^{\varphi^\lambda} \left(f,   \delta_{n, \lambda}(x)  \right),
\]
where, for $n\in\N$ and $0\leq \lambda<2$,
\[
\delta_{n, \lambda}(x) :=
\begin{cases}
\left[ n^{-1}\varphi(x)  \right]^{1-\lambda/2 }, & \mbox{\rm if }\; x\in\left[0, n^{-2}\right]\cup \left[1-n^{-2},1\right] , \\
n^{-1} \varphi^{1-\lambda}(x) ,  & \mbox{\rm if }\; n^{-2}< x <1-n^{-2},
\end{cases}
\]
and implies that, for $f\in C[0,1]$ with $f'\in\AC_\loc(0,1)$ and $\norm{\varphi^{2\lambda} f''} < \infty$,
\[
|f(x)-M_n(f, x)| \leq c \left[\delta_{n, \lambda}(x)\right]^2  \norm{\varphi^{2\lambda} f''}, \quad x\in [0,1].
\]
\end{remark}

Throughout this paper, we use the notation  $e_i(x):= x^i$, $i\in\N_0$, and
 $(\beta)_k := \beta(\beta+1)\dots (\beta+k-1)$ for $k\geq 1$, and $(\beta)_0 := 1$
(\ie  $(\beta)_k$ is the Pochhammer function).

\sect{Approximation by positive linear operators preserving linear functions}

Recall that an operator  $L: C[0,1] \mapsto C[0,1]$ is positive if $L(f, x) \geq 0$ for all $x\in [0,1]$ provided $f(x)\geq 0$, $x\in [0,1]$.

 Let
\[
 \Omega  := \left\{ \psi \in  C[0,1] \st \psi(x) > 0, \; 0<x<1, \;\;  \mbox{\rm and $\psi^2$ is concave on $[0,1]$} \right\}
\]
and
 \[
K_{2,\psi}(f,h^2) := \inf_{g'\in\AC_\loc(0,1)}(\|f-g\|+h^2\|\psi^2  g''\|) .
\]

The following lemma is a corollary of a more general theorem  \cite{felten}*{Theorem 1} that was proved for positive linear operators preserving constants.

\begin{lemma}[Felten~\cite{felten}] \label{fel}
Suppose that $\psi \in \Omega$
 and
 $L:  C[0,1] \mapsto  C[0,1]$ is a positive linear operator preserving linear functions (\ie $L(e_i)=e_i$, $i=0,1$). Then, for any $f\in  C[0,1]$ and $x\in (0,1)$, one has
\[
 |f(x)-L(f, x)| \leq 4 K_{2, \psi} \left( f, \frac{L(e_2, x)-x^2}{\psi^2(x)} \right) .
\]
\end{lemma}

\begin{lemma}[Bustamante~\cite{busta}*{Theorem 11}] \label{busta}
Suppose that $\psi \in \Omega$
 and
 $L:  C[0,1] \mapsto  C[0,1]$ is a positive linear operator preserving linear functions. Then, for any $f\in  C[0,1]$ and $x\in (0,1)$, one has
 \[
 |f(x)-L(f, x)| \leq \left( \frac 32 + \frac{3}{2 h^2 \psi^2(x)} \left( L(e_2, x)-x^2\right) \right) \omega_2^\psi(f, h).
 \]
\end{lemma}
If one does not worry about the constants then \lem{busta} follows from \lem{fel} provided that $\psi$ is such that
$K_{2, \psi}(f, h^2) \leq c  \omega_2^\psi(f, h)$.

Since
\[
\varphi^\lambda  \in\Omega  \quad \mbox{\rm if and only if}\quad  0\leq \lambda\leq 1,
\]
we  conclude that Lemmas~\ref{fel} and \ref{busta} hold  for $\psi := \varphi^\lambda$ with $0\leq \lambda\leq 1$.

We will now provide a rather elementary proof  that a similar statement (we do not worry about constants) is valid for all $0\leq \lambda < 2$ (for $1<\lambda <2$ this seems to be a new result).

\begin{lemma} \label{mainlemma}
If $L:  C[0,1] \mapsto  C[0,1]$ is a positive linear operator preserving linear functions,  then for any $0\leq \lambda<2$, $f\in  C[0,1]$, $x \in[0,1]$, $\xxx\in (0,1)$ and $h>0$, one has
 \be \label{mainest}
 |f(x)-L(f, x)| \leq \left(  2 + \frac{4}{2-\lambda} \cdot \frac{ L(e_2, x) - x^2 + 2 (x-\xxx)^2 }{h^2 \varphi^{2\lambda}(\xxx)} \right) K_{2,\varphi^\lambda}(f,h^2) .
 \ee
\end{lemma}

\begin{proof}  We first show that for any $g\in C[0,1]$ such that $g'\in \AC_\loc(0,1)$,
\be \label{auxbusta}
\left|g(t)-g(\xxx)-(t-\xxx)g'(\xxx)\right|\le\frac{4}{2-\lambda} \frac{(t-\xxx)^2}{\varphi^{2\lambda}(\xxx)}   \norm{\varphi^{2\lambda} g''}   ,
\ee
for all $\xxx\in(0,1)$ and $t\in[0,1]$.

Since $g'\in \AC_\loc(0,1)$ we have
\[
\left|g(t)- g(\xxx)-(t-\xxx)g'(\xxx) \right|=\left|\int_\xxx^t(t-u)g''(u)du\right|  \le  \norm{\varphi^{2\lambda}g''}  \int_t^\xxx\frac{u-t}{\varphi^{2\lambda}(u)}du.
\]
Without loss of generality, assume that  $\xxx\in (0,1/2]$.  If $\xxx/2\le t\le 1-\xxx/2$, then $\varphi (u)\ge \varphi(\xxx/2)\ge 2^{-1/2} \varphi(\xxx)$ for any $u$ between $t$ and $\xxx$,
and so
\[
\int_{t}^\xxx\frac{u-t}{\varphi^{2\lambda}(u)}\,du \le \frac{4}{\varphi^{2\lambda}(\xxx)}\int_{t}^\xxx  (u-t)\,du=2\frac{(t-\xxx)^2}{\varphi^{2\lambda}(\xxx)}.
\]
If $0\le t<\xxx/2$, then
\begin{eqnarray*}
\int_{t}^\xxx\frac{u-t}{\varphi^{2\lambda}(u)}\,du & \le & \int_{0}^\xxx\frac{u}{\varphi^{2\lambda}(u)}\,du
\le \frac1{(1-\xxx)^{\lambda}}\int_{0}^\xxx    u^{1-\lambda} \, du
=\frac1{2-\lambda}\frac{\xxx^2}{\varphi^{2\lambda}(\xxx)} \\
&\le & \frac4{2-\lambda}\frac{(\xxx-t)^2}{\varphi^{2\lambda}(\xxx)} .
\end{eqnarray*}
For the remaining case $1-\xxx/2<t \leq 1$, the proof is exactly the same, and so \ineq{auxbusta} is verified.

Since $L$ is positive we conclude that, for any functions $F,G\in C[0,1]$ such that $|F(t)| \leq G(t)$, $t\in[0,1]$, the inequality
$|L(F, x)|\leq L(G, x)$ is valid for all $x\in[0,1]$. Applying this observation to \ineq{auxbusta} and recalling that $L$ is linear and preserves linear functions we immediately get
\[
|L(g,x) - g(\xxx)-(x-\xxx)g'(\xxx)| \leq \frac{4}{2-\lambda} \frac{\norm{\varphi^{2\lambda} g''}}{\varphi^{2\lambda}(\xxx)} \left( L(e_2, x)-2x \xxx + \xxx^2 \right) , \quad x\in [0,1].
\]
Together with \ineq{auxbusta}   (with $t$ replaced by $x$) this yields
\begin{eqnarray*}
|L(g,x)-g(x)| & \leq &    |L(g,x) - g(\xxx)-(x-\xxx)g'(\xxx)| + |g(x) - g(\xxx)-(x-\xxx)g'(\xxx)| \\
& \leq &
\frac{4}{2-\lambda} \frac{\norm{\varphi^{2\lambda} g''}}{\varphi^{2\lambda}(\xxx)}  \left( L(e_2, x)-2x \xxx + \xxx^2 \right) +
\frac{4}{2-\lambda} \frac{\norm{\varphi^{2\lambda} g''}}{\varphi^{2\lambda}(\xxx)} (x-\xxx)^2 \\
& = &
\frac{4}{2-\lambda} \frac{\norm{\varphi^{2\lambda} g''}}{\varphi^{2\lambda}(\xxx)} \left(  L(e_2, x) - x^2 + 2 (x-\xxx)^2 \right) .
\end{eqnarray*}

Suppose now that, for each $\eps>0$, $\geps\in C[0,1]$ with $\geps' \in\AC_\loc(0,1)$ is such that
\[
\|f-\geps \|+h^2\|\varphi^{2\lambda} \geps ''\| \leq K_{2,\varphi^\lambda}(f,h^2) + \eps .
\]
Taking into account that any positive linear  operator $L$  preserving constants is a contraction (\ie $|L(F,x)| \leq \norm{F}$)     we have
\begin{eqnarray*}
|f(x)-L(f,x)| &\leq&  |f(x)-\geps(x)|+|\geps(x)-L(\geps,x)|+|L(\geps-f,x)|\\
&\leq &
2 \norm{f-\geps}
+ \frac{4}{2-\lambda} \frac{\norm{\varphi^{2\lambda} \geps''}}{\varphi^{2\lambda}(\xxx)}\left(  L(e_2, x) - x^2 + 2 (x-\xxx)^2 \right) \\
&\leq&
\left(  2 + \frac{4}{2-\lambda} \cdot \frac{ L(e_2, x) - x^2 + 2 (x-\xxx)^2 }{h^2\varphi^{2\lambda}(\xxx)} \right) (K_{2,\varphi^\lambda}(f,h^2) + \eps) ,
\end{eqnarray*}
and \ineq{mainest} follows.
\end{proof}

\begin{remark}
Clearly, \lem{mainlemma} remains valid if $\varphi^\lambda$ is replaced by a function $\phi$ such that, for $\xxx\in(0,1)$ and   $t\in[0,1]$,
 \be \label{glin}
 \int_t^\xxx\frac{u-t}{\phi^{2}(u)}du \leq c \frac{(t-\xxx)^2}{\phi^{2}(\xxx)}.
 \ee
In particular, this inequality is satisfied if $\phi$ is such that
 \begin{enumerate}[(i)]
\item
 $x^{-\beta}\phi(x)$ and $(1-x)^{-\beta}\phi(x)$  are, respectively, quasi decreasing and quasi increasing on $(0,1)$ for some $\beta < 1$ ($g$ is quasi decreasing if $g(x) \geq c g(y)$ for $x \leq y$ for some absolute constant $c$; $g$ is quasi increasing if $-g$ is quasi decreasing), and
 \item $\phi(x) \geq c \max\{\phi(\epsilon),\phi(1-\epsilon)\}$, for any  $0\leq \epsilon \leq 1/2$ and $\epsilon \leq x \leq 1-\epsilon$.
  \end{enumerate}
For example, any $\phi$ such that  $\phi(x)\sim\phi(1-x)$ and  $\phi^2$ is concave on $[0,1]$ satisfies these conditions.  Note also that \ineq{glin} is not valid for
$\phi(x) = \varphi^2(x)$ (which is concave on $[0,1]$) and so we cannot replace the inequality ``$\beta <1$'' in (i) by ``$\beta\leq 1$''.
\end{remark}

Note that if $L: C[0,1]\mapsto C[0,1]$ is a positive linear operator preserving linear functions, then $L(f, 0)=f(0)$ and $L(f, 1)=f(1)$ for any $f\in C[0,1]$. Indeed, suppose that it is not the case and, without loss of generality, assume that $\epsilon := L(f, 0) - f(0)>0$ for some $f\in C[0,1]$. Continuity of $f$  implies that there exists $m\in\R$ depending on $f$ such that $l(x):= mx + L(f, 0)$ satisfies $l(x) \geq f(x) + \epsilon/2$ for all $x\in [0,1]$.
(For example, one can choose $m := 2\norm{f}/\delta$ where $\delta>0$ is such that $|f(x)-f(0)| <\epsilon/2$ for $0\leq x \leq \delta$.)
Then $l(x) = L(l, x) \geq L(f, x) + \epsilon/2$ and letting $x=0$ we get a contradiction.

The above observation implies that, if  $L_n: C[0,1] \mapsto\Poly_n$ is a sequence of positive linear polynomial operators preserving linear functions and such that
$L_n(f, \cdot)\in\Poly_2$ provided $f\in\Poly_2$,
then $L_n(e_2, x)  = x^2+\a_n \varphi^2(x)$, $\alpha_n>0$.
Taking into account a well known fact that
 $K_{2,\varphi^\lambda}(f,h^2) \sim  \omega_2^{\varphi^\lambda} (f, h)$, for $0<h\leq c_0$ (see \cite{dt}*{Theorem 2.1.1}), we immediately have the following consequence of \lem{mainlemma} by setting
 \[
 \xi :=
 \begin{cases}
 x , & \mbox{\rm if }\; \b_n \leq x \leq 1-\b_n, \\
 x + \sqrt{\b_n}\varphi(x) , & \mbox{\rm if }\; 0\leq x < \b_n , \\
 x- \sqrt{\b_n}\varphi(x), & \mbox{\rm if }\; 1-\b_n<x \leq 1 ,
 \end{cases}
 \]
where $\b_n := \min\{\a_n, 1/4\}$, and noting that
\begin{eqnarray*}
  \frac{ \a_n \varphi^2(x) + 2 (x-\xxx)^2 }{ \varphi^{2\lambda}(\xxx)}
&\leq&
 \begin{cases}
\a_n \varphi^{2-2\lambda}(x), & \mbox{\rm if }\; x\in [\b_n,  1-\b_n], \\
 12 \a_n \b_n^{-\lambda/2} \varphi^{2-\lambda}(x), & \mbox{\rm if }\; x\in (0, \b_n)\cup (1-\b_n,1),\\
 \end{cases}
\\
& \leq&
 {50 \a_n \varphi^{2-\lambda}(x) \over \left( \varphi(x) + \sqrt{\b_n} \right)^\lambda }.
\end{eqnarray*}

\begin{corollary} \label{maincor}
If $L_n:  C[0,1] \mapsto \Poly_n$ is a sequence of positive linear polynomial operators preserving linear functions,  then for any $0\leq \lambda< 2$, $f\in  C[0,1]$, $x\in [0,1]$ and $0<h\leq c_0$, one has
 \be \label{ineq:maincor}
 |f(x)-L_n(f, x)| \leq c \left(  1 +  {\a_n \varphi^{2-\lambda}(x) \over h^2 \left( \varphi(x) + \sqrt{\min\{\a_n, 1/4\}} \right)^\lambda }  \right) \omega_2^{\varphi^\lambda} (f, h),
 \ee
 where $\a_n>0$ is such that $L_n(e_2, x) - x^2 = \a_n \varphi^2(x)$,
$c_0$ is some absolute constant, and the constant  $c$ depends on $\lambda$ as $\lambda \to 2-$.
\end{corollary}

\begin{remark}
Estimate \ineq{ineq:maincor} implies the following weaker inequality
\[
|f(x)-L_n(f, x)| \leq c \left(  1 +  {\a_n \varphi^{2-2\lambda}(x) \over h^2   }  \right) \omega_2^{\varphi^\lambda} (f, h)
\]
which, in turn, yields
\[
|f(x)-L_n(f, x)| \leq c  \omega_2^{\varphi^\lambda} \left( f, \sqrt{\a_n} \, \varphi^{1-\lambda}(x) \right) .
\]
\end{remark}

In the next section, we discuss some applications for the classical Bernstein polynomials (clearly, similar results can be stated for many other positive linear polynomial operators) and show how our estimates can be used for $\lambda\in(1,2)$.

\subsection{Some applications for Bernstein polynomials.} \label{whyuseful}

Let
\[
p_{n,k}(x):={\binom nk} x^k(1-x)^{n-k},\quad 0\le k\le n,
\]
be the Bernstein fundamental polynomials, and recall that the classical Bernstein operator
\[
B_n(f, x) := \sum_{k=0}^{n} f(k/n)p_{n,k}(x)
\]
 is positive, linear, preserves linear functions and
$B_n(e_2, x) - x^2 = \varphi^2(x)/n$. \cor{maincor} (with $\a_n = 1/n$) implies the following result.

\begin{corollary} \label{clBern}
If $n\in\N$ and $B_n: C[0,1]\mapsto \Poly_n$ is the classical Bernstein polynomial, then, for any $0\leq \lambda< 2$, $f\in  C[0,1]$, $x\in [0,1]$ and $0<h\leq c_0$, one has
 \be \label{ineq:corbern}
 |f(x)-B_n(f, x)| \leq c \left(  1 +  {\varphi^{2-\lambda}(x) \over h^2 n \left( \varphi(x) + n^{-1/2}  \right)^\lambda }  \right) \omega_2^{\varphi^\lambda} (f, h),
 \ee
 where $c_0$ is some absolute constant, and the constant  $c$ depends on $\lambda$ as $\lambda \to 2-$.
 In particular,
 \be \label{inbernstein}
  |f(x)-B_n(f, x)| \leq c  \omega_2^{\varphi^\lambda} \left( f, \gamma_{n, \lambda}(x)   \right) ,
 \ee
 where
\begin{eqnarray*}
\gamma_{n, \lambda}(x) &:=&  n^{-1/2} \varphi^{1-\lambda/2}(x) \left(\varphi(x) + n^{-1/2} \right)^{-\lambda/2} \\
& \sim &
\begin{cases}
\left[ n^{-1}x (1-x) \right]^{(2-\lambda)/4 }, & \mbox{\rm if }\; x\in\left[0, n^{-1}\right]\cup \left[1-n^{-1},1\right] , \\
n^{-1/2} [x(1-x)]^{(1-\lambda)/2}  ,  & \mbox{\rm if }\; n^{-1}< x <1-n^{-1} .
\end{cases}
\end{eqnarray*}
\end{corollary}

\begin{remark}
Clearly, $\gamma_{n, \lambda}(x) \leq n^{-1/2} \varphi^{1-\lambda}(x)$ and so \ineq{inbernstein} immediately implies
\be \label{classdi}
  |f(x)-B_n(f, x)| \leq c \omega_2^{\varphi^\lambda} \left( f, n^{-1/2} \varphi^{1-\lambda}(x) \right) ,
\ee
which is the main result of \cite{d-jat} in the case $0\leq \lambda\leq 1$.
\end{remark}

\begin{remark}
For $\lambda=1$, \ineq{inbernstein} becomes
 \be \label{intach}
  |f(x)-B_n(f, x)| \leq c \omega_2^{\varphi} \left(f, n^{-1/2} \sqrt{ {\varphi(x) \over \varphi(x) + n^{-1/2} } }  \right) ,
 \ee
which is equivalent to \cite{tach}*{Theorem 1.1}.
\end{remark}

We will now  consider a very simple example in order to compare the estimates produced by different methods.

Suppose that one wants to know how well Bernstein polynomials approximate the function $\fep (x):= x^\ep$, $0<\ep <1$.
One can easily calculate (see also \cite{dt}*{Section 3.4}) that, for $0\leq \lambda\leq 2$,
 \[
 \omega_2^{\varphi^\lambda} (\fep, t) \sim
 \begin{cases}
 t^2, & \mbox{\rm if }\; \ep + \lambda - 2 \geq 0, \\
 t^{\ep/(1-\lambda/2)}, & \mbox{\rm if }\; \ep + \lambda - 2 < 0.
 \end{cases}
 \]
The classical results (estimate \ineq{classdi} for $\lambda=0$ and $\lambda=1$)   immediately yield
\be \label{classics}
|\fep(x)-B_n(\fep, x)| \leq c \left(n^{-1/2}  \varphi(x) \right)^\ep \andd
\norm{\fep - B_n(\fep, \cdot)} \leq c n^{-\ep} .
\ee
Using \ineq{classdi} for $0\leq \lambda\leq 1$, we may conclude  that
\be \label{tmp1}
|\fep(x)-B_n(\fep, x) | \leq c \left( n^{-1/2}  \varphi^{1-\lambda}(x)   \right)^{\ep/(1-\lambda/2)} ,
\ee
but this is not  better  than \ineq{classics} since, for all  $x,\lambda\in [0,1]$,
\[
\min\left\{ n^{-1/2}  \varphi(x), n^{-1} \right\} \leq   \left( n^{-1/2}  \varphi^{1-\lambda}(x)   \right)^{1/(1-\lambda/2)}.
\]
However, if we  choose $\lambda = 2-\ep$ (note that $1<\lambda<2$), then  $\omega_2^{\varphi^\lambda} (\fep, t) \sim t^2$,  and \ineq{inbernstein} yields
 \[
|\fep(x)-B_n(\fep, x) | \leq c n^{-1} {\varphi^{\ep}(x)  \over \left(\varphi(x) + n^{-1/2} \right)^{2-\ep} } .
\]
This implies
\be \label{tmpp}
|\fep(x)-B_n(\fep, x) | \leq c
\begin{cases}
 n^{-1}    \varphi^{2\ep-2}(x)  , & \mbox{\rm if }\; x\in [1/n,1-1/n] ,\\
\left(n^{-1/2}  \varphi(x) \right)^\ep , & \mbox{\rm if }\; x\in [0, 1/n)\cup (1-1/n] ,
\end{cases}
\ee
 which is better in the middle of $[0,1]$ than anything that one can get from \ineq{classics} or \ineq{tmp1}.
Now, the classical Voronovskaya theorem yields
\be \label{voron}
\lim_{n\to\infty} n\left( \fep(x)- B_n(\fep, x) \right) = -  {\varphi^2(x) \over 2} \fep''(x) = {\ep(1-\ep) \over 2} x^{\ep-2} \varphi^2(x) ,
\ee
and this implies that \ineq{tmpp} in the middle of $[0,1]$  cannot be improved (note that \ineq{voron} actually implies \ineq{tmpp} in the middle of $[0,1]$ for sufficiently large $n$ depending on $x$).

This elementary example illustrates that it is sometimes advantageous to work with moduli $\omega_2^{\varphi^\lambda}$ with $\lambda$'s greater than $1$.

\subsection{Genuine Bernstein-Durrmeyer operator}

Let
 $U_n: C[0,1]\mapsto\Poly_n$, $n\ge2$, be defined by
\[
 U_n (f, x):=f(0)(1-x)^n+f(1)x^n+(n-1)\sum_{k=1}^{n-1}p_{n,k}(x)\int_0^1p_{n-2,k-1}(t)f(t)dt.
\]
 It seems that operators $U_n$   were first considered by Goodman and Sharma in \cite{gs} (see \cite{gkr} for further discussions of the history of these operators as well as different names used for them in the literature).

Clearly, $U_n$ are positive linear operators with  $U_n (f, 0)=f(0)$ and $U_n (f, 1)=f(1)$.
Also,   it immediately follows from the following lemma that
 \be \label{opun}
U_n(e_0, x) = 1, \quad  U_n(e_1, x) = x \andd      U_n (e_2, x) = x^2 + \frac{2x(1-x)}{n+1} ,
 \ee
 and so operators $U_n$ preserve linear functions.

 \begin{lemma}
For any $n\geq 2$,
\be \label{mainun}
  U_n (e_i, x) =
{(n-1)! \, i! \over (n+i-1)!} \sum_{j=\max\{0, i-n\}}^{i-1}  {i-1 \choose j} {n \choose i-j}   x^{i-j},   \quad   i\geq 1,
\ee
and  $U_n (e_0, x) = 1$.
 \end{lemma}

\begin{proof} The proof is standard and is based on the fact that, for any $i\geq 0$, $n\geq 0$ and $0\leq k\leq n$,
\[
\int_0^1 p_{n,k}(t) e_i(t) dt = {(k+1)_i \over (n+1)_{i+1}} \andd (k)_i \,x^k   =  x \cdot \left. {d^i \over dy^i} y^{k+i-1}\right|_{y=x}.
\]
We omit details.
%
%
\end{proof}

\begin{remark}
The following identity can also be used to calculate   $U_n(e_k, x)$:
 \[
 U_n(e_{k+1}, x) = {(n-k)x +2k \over n+k} U_n(e_k, x) - {k(k-1)(1-x) \over (n+k)(n+k-1)} U_n(e_{k-1}, x) .
 \]
 \end{remark}

\subsection{Gavrea's  operator}

In this  section, we discuss several properties of the operator $H_{n+2}$ that was introduced by Gavrea \cite{g}.
Everything here follows from \cites{g, cggkz}, and we include this  section in the current manuscript only for readers' convenience (we also somewhat clean up some of the proofs making them, in our opinion, more transparent by utilizing the notation \ineq{in:positive3} and \cor{keycor3}).

For any $n\in\N$ and a fixed (generating) polynomial $\gP_n(x) =\sum_{k=0}^n a_k x^k$, Gavrea's operator $H_{n+2}: C[0,1]\mapsto\Poly_{n+2}$ is defined as
\be \label{gavreaop}
  H_{n+2}(\gP_n; f, x):=    \sum_{k=0}^n\frac{a_k}{k+1}U_{k+2}(f, x).
\ee
Clearly, these operators are linear. It turns out that they are also positive and, moreover, preserve monotonicity of high orders if a generating polynomial $\gP$ satisfies certain properties (see \lem{mainspa}).

By \ineq{opun} we immediately get
\[
H_{n+2}(\gP_n; e_0, x) = \sum_{k=0}^n\frac{a_k}{k+1} = \int_0^1 \gP_n(t) dt ,
\]
\[
H_{n+2}(\gP_n; e_1, x) = \sum_{k=0}^n\frac{a_k}{k+1} x = x \int_0^1 \gP_n(t) dt
\]
and
\[
H_{n+2}(\gP_n; e_2, x) = \sum_{k=0}^n\frac{a_k}{k+1} \left(x^2 + \frac{2x(1-x)}{k+3} \right) .
\]
Hence,
\begin{eqnarray*}
 H_{n+2}(\gP_n; e_2, x) &-& x^2 \int_0^1 \gP_n(t) dt
 \;  = \;
x(1-x) \sum_{k=0}^n \left( \frac{a_k}{k+1} - \frac{a_k}{k+3}  \right)   \\
&  = &
x(1-x) \left( \int_0^1 \gP_n(t) dt  - \int_0^1 t^2 \gP_n(t) dt \right) .
\end{eqnarray*}

It was shown in \cite{g}*{Lemma 3} that, for all $0<x<1$ and $n\geq 2$,
\begin{eqnarray} \label{in:gav} \nonumber
\lefteqn{ U_n(f, x)  =  f(0)(1-x)^n+f(1)x^n }\\
& + &
 (n-1)(1-x)^n \int_0^x {D_{n-2}(f, y) \over (1-y)^n} dy +
 (n-1)x^n \int_x^1 {D_{n-2}(f, y) \over y^n} dy ,
\end{eqnarray}
where
\be \label{bd}
 D_n (f, x):= (n+1)\sum_{k=0}^{n} p_{n,k}(x) \int_0^1p_{n,k}(t)f(t)dt
\ee
is the (usual) Bernstein-Durrmeyer operator (see also \rem{rem32}).

Note that \ineq{in:gav} follows from the identity
\[
{ 1 \over n-1} p_{n, k+1}(x) = \int_0^x \left( \frac{1-x}{1-y} \right)^n p_{n-2, k}(y) dy
+ \int_x^1 \left( \frac{x}{y} \right)^n p_{n-2, k}(y) dy,
\]
which is valid for $0\leq k \leq n-2$ and is easily verified directly.

Now,  \cor{keycor3} yields
\begin{eqnarray*}
\lefteqn{ H_{n+2}(\gP_n; f, x)   =   f(0) \sum_{k=0}^n\frac{a_k}{k+1} (1-x)^{k+2} + f(1) \sum_{k=0}^n\frac{a_k}{k+1} x^{k+2} } \\
&& + \int_0^x \left({1-x \over 1-y}\right)^2 \sum_{k=0}^n a_k \left({1-x \over 1-y}\right)^k D_{k}(f, y) dy   \\
 && + \int_x^1 \left({x\over y}\right)^2 \sum_{k=0}^n a_k \left({x\over y}\right)^k D_{k}(f, y) dy \\
 &=&
 f(0) (1-x) \int_0^{1-x} \gP_n(y) dy + f(1) x \int_0^x \gP_n(y) dy \\
 && + \int_0^x \left({1-x \over 1-y}\right)^2 \left[ L_n^{\langle 0 \rangle} \left(\gP_n, {1-x \over 1- (\cdot)}, 1, 0, [0,x] \, ; f,y\right) \right.\\
 && \left. \mbox{\hspace{3cm}} +
 L_n^{\langle 0 \rangle} \left(\gP_n, {1-x \over 1- (\cdot)}, 0, 0, [0,x] \, ; f,y\right) \right] dy \\
 && +
 \int_x^1 \left({x\over y}\right)^2 \left[ L_n^{\langle 0 \rangle} \left(\gP_n, {x \over (\cdot)}, 1, 0, [x,1] \, ; f,y \right) \right. \\
 && \left. \mbox{\hspace{3cm}} +
  L_n^{\langle 0 \rangle} \left(\gP_n, {x \over (\cdot)}, 0, 0, [x,1] \, ; f,y \right) \right] dy ,
\end{eqnarray*}
 which implies that the operator $H_{n+2}$ is positive provided $\gP_n(x)\geq 0$ and $\gP_n'(x)\geq 0$ for all $x\in [0,1]$.

 Now, using the fact that $ \frac{d}{dx} U_{n+1}(f, x) = D_n(f', x)$ for any $n\in \N_0$  (the proof of this is straightforward or see \cite{cggkz}*{Theorem 12})
 by virtue of \lem{lem:deriv} (see also \rem{rem32}) we conclude that, for any $\nu\in\N$, $f\in C^{\nu}[0,1]$ and $k\geq \nu-2$,
\[
\frac{d^\nu}{dx^\nu} U_{k+2}(f, x)  = \frac{d^{\nu-1}}{dx^{\nu-1}} D_{k+1}^{\langle 0 \rangle}(f', x)
= {(k+1)! \over (k-\nu+2)! (k+3)_{\nu-1}} D^{\langle \nu-1 \rangle}_{k-\nu+2}\left( f^{(\nu)}, x \right) .
\]
Recalling that $U_{k+2}(f,\cdot) \in \Poly_{k+2}$, this implies, for $\nu\geq 2$,
\begin{eqnarray*}
\frac{d^\nu}{dx^\nu} H_{n+2}(\gP_n; f, x) & = & \sum_{k=\nu-2}^n\frac{a_k}{k+1} \cdot {(k+1)! \over (k-\nu+2)! (k+3)_{\nu-1}} D^{\langle \nu-1 \rangle}_{k-\nu+2}\left( f^{(\nu)}, x \right) \\
&=&
\sum_{k=0}^{n-\nu+2} {(k+\nu-2)! \over k! (k+\nu+1)_{\nu-1}} \, a_{k+\nu-2} D^{\langle \nu-1 \rangle}_{k}\left( f^{(\nu)}, x \right) \\
& = &
\frac{1}{(\nu)_\nu} \sum_{k=0}^{n-\nu+2} {(\nu)_k (k+1)_{\nu-2} \over (2\nu)_k}  (k+\nu) \, a_{k+\nu-2} D^{\langle \nu-1 \rangle}_{k}\left( f^{(\nu)}, x \right) .
\end{eqnarray*}
  Since $(k+\nu) (k+1)_{\nu-2} = (k)_{\nu-1}+ \nu (k+1)_{\nu-2}$, using \cor{keycor3} we write
\begin{eqnarray*}
\frac{d^\nu}{dx^\nu} H_{n+2}(\gP_n; f, x)  & = &
\frac{1}{(\nu)_\nu} \left[
\nu L_n^{\langle \nu-1 \rangle} (\gP_n, 1, \nu-2, \nu-2, [0,1] \, ; f^{(\nu)},x) \right. \\
&& \left. +
L_n^{\langle \nu-1 \rangle} (\gP_n, 1, \nu-1, \nu-2, [0,1] \, ; f^{(\nu)},x) \right] ,
\end{eqnarray*}
and conclude that $\frac{d^\nu}{dx^\nu} H_{n+2}(\gP_n; f, x) \geq 0$ provided $f^{(\nu)} (x)\geq 0$, $\gP_n^{(\nu-1)}(x)\geq 0$ and $\gP_n^{(\nu-2)}(x)\geq 0$ on $[0,1]$.

In the case $\nu=1$, we have
\begin{eqnarray*}
\frac{d}{dx} H_{n+2}(\gP_n; f, x) & = & \sum_{k=0}^n\frac{a_k}{k+1}  D^{\langle 0 \rangle}_{k+1}\left( f', x \right)
=
\sum_{k=1}^{n+1} {a_{k-1}\over k}  D^{\langle 0 \rangle}_{k}\left( f', x \right) \\
& = &
\sum_{k=0}^{n+1} {b_k\over k+1}  D^{\langle 0 \rangle}_{k}\left( f', x \right)
 =
L_{n+1}^{\langle 0 \rangle} (\widetilde\gP_{n+1}, 1, 0, 0, [0,1]; f', x) ,
\end{eqnarray*}
where $b_0 :=0$ and $b_k := (k+1)a_{k-1}/k$, $1\leq k \leq n+1$, and
\[
\widetilde\gP_{n+1}(x) := \sum_{k=0}^{n+1} b_k x^k  = x \gP_n(x)+ \int_0^x \gP_n(y) dy .
\]
\cor{keycor3} now implies that $\frac{d}{dx} H_{n+2}(\gP_n; f, x) \geq 0$ provided $f'(x)\geq 0$ and $\widetilde\gP_{n+1}(x) \geq 0$ on $[0,1]$ (and nonnegativity of $\gP_n$ on $[0,1]$ is clearly sufficient for the latter inequality).

We summarize the above   discussions in this section in the following lemma.

\begin{lemma}[\cite{g}*{Theorem 2} and \cite{cggkz}*{Theorem 14}]
 \label{mainspa}
Let $r, n\in\N$ and suppose that a generating polynomial $\gP_n\in\Poly_n$ is such that
\begin{enumerate}[(i)]
\item for all $0\leq \nu \leq r$, $\gP_n^{(\nu)} (x) \geq 0$, $x\in [0,1]$,
\item $\ds \int_0^1 \gP_n(t) dt = 1$.
\end{enumerate}
Then the operator $H_{n+2}: C[0,1]\mapsto\Poly_{n+2}$ defined in \ineq{gavreaop} has the following properties
\begin{enumerate}[(i)]
\item $H_{n+2}$ is a positive linear operator preserving linear functions, \ie
$H_{n+2}(\gP_n; g, \cdot) = g$ for any $g\in\Poly_1$,
\item $\ds H_{n+2}(\gP_n; e_2, x) = x^2 + x(1-x) \left( 1  - \int_0^1 t^2 \gP_n(t) dt \right)$,
\item For every $0\leq k \leq r+1$, $H_{n+2}$ is $k$-monotonicity preserving. In other words, if $f\in\Delta^{(k)}$, then
$H_{n+2}(\gP_n; f, \cdot)\in\Delta^{(k)}$.
\end{enumerate}

\end{lemma}

\subsection{A particular generating polynomial}

Let $T_m(x):=\cos m\arccos x$, $x\in[-1,1]$,
be the Chebyshev polynomial of degree $m$, $\tilde x=\cos(\pi/2m)$ be its  rightmost  zero, $ x_1=\cos(\pi/m)$
be its rightmost local minimum, $I_1:=[x_1,1]$ (its length $|I_1|=1-x_1=2\sin^2(\pi/2m)$). Then
$$
\tau_m(x):=\frac{T_m(x)}{x-\tilde x}|I_1|,
$$
is a polynomial of degree $m-1$. It is well known (see, \eg \cite{kls}*{Appendix A}) and is not difficult to check, that
\be\label{kls}
\frac43<\tau_m(x)<4,\quad x\in I_1.
\ee

Also note that  since $|I_1|< 2 (\tilde x-x_1)$, we have
\be\label{kls1}
|\tau_m(x)|\le\frac{|I_1|}{\tilde x-x}< \frac{2 |I_1|}{1-x},\quad x\in[-1,1]\setminus I_1.
\ee

\begin{lemma}  \label{lem210}
For each $r\in\N$ and $n\in\N_0$,  there exists a polynomial $P_n$ of degree $\le n$ such that,
for every $0\leq \nu\leq r $,
\be\label{mon}
 P_n^{(\nu)}(x)\ge0, \quad x\in[0,1],
\ee
\be\label{1}
\int_0^1P_n(x)dx=1,
\ee
and
\be\label{small}
1-\int_0^1 x^\mu P_n(x)dx\le\frac c{n^2}, \quad \mu\in\N ,
\ee
where $c$ is a   constant that depends only on $r$ and $\mu$.
\end{lemma}

We remark that the estimate \ineq{small} cannot be improved. An indirect proof of this fact is that if we could improve it for $\mu=2$ and some  polynomial $P_n$ satisfying \ineq{mon} and \ineq{1}, then a well known Korovkin's result on approximation by positive linear operators would be violated by $H_{n+2}(P_n; f, \cdot)$ (since we would have $H_{n+2}(P_n; e_i, x) = o(n^{-2})$ for $i=0,1,2$). One can also easily prove this fact directly. Indeed, let $P_n$ be an arbitrary polynomial for $\Poly_n$ such that $P_n(x)\geq 0$, $x\in [0,1]$, and \ineq{1} is satisfied.
Then, for any $\mu\geq 1$,
\begin{eqnarray*}
1-\int_0^1 x^\mu P_n(x)dx & = & \int_0^1 (1- x^\mu)  P_n(x)dx \geq \int_0^1 (1- x)  P_n(x)dx \\
&\geq & \int_0^{1-n^{-2}} (1- x)  P_n(x)dx \geq n^{-2} \int_0^{1-n^{-2}}  P_n(x)dx \\
& \geq & c n^{-2} \int_0^{1}  P_n(x)dx = c n^{-2} ,
\end{eqnarray*}
 where the last inequality follows from a well known Remez inequality for algebraic polynomials in $L_1$ (see, \eg \cite{be}*{Theorem A.4.10}).

\begin{proof}[Proof of \lem{lem210}]  Clearly, it is enough to prove this lemma for $n> 8 r$.
Let $Q_{n-r}$ be a nonnegative (on $[0,1]$) polynomial of degree $\leq n-r$, and define
\[
P_n(x):=\lambda_n\int_0^x(x-t)^{r-1}Q_{n-r}(t)dt.
\]
The polynomial $P_n$ satisfies \ineq{mon} and  since
\be\label{pnn}
\int_0^1P_n(x)dx=\lambda_n\int_0^1\int_0^x(x-t)^{r-1}Q_{n-r}(t)dtdx=\frac{\lambda_n}r\int_0^1(1-t)^rQ_{n-r}(t)dt,
\ee
  in order for \ineq{1} to hold, we need to take
\be\label{lambda}
\lambda_n:= r \left( \int_0^1(1-t)^rQ_{n-r}(t)dt \right)^{-1}.
\ee
Now,
\begin{eqnarray*}
\int_0^1 x^\mu P_n(x)dx &=& \lambda_n\int_0^1 x^\mu \int_0^x(x-t)^{r-1}Q_{n-r}(t)dtdx\\
&=& \lambda_n\int_0^1Q_{n-r}(t)\int_t^1 x^\mu (x-t)^{r-1}dxdt .
\end{eqnarray*}
Since
\begin{eqnarray*}
\int_t^1 x^\mu (x-t)^{r-1}dx & = & \int_t^1 \sum_{i=0}^\mu {\mu \choose i} (x-t)^{i+r-1} t^{\mu - i} \, dx \\
& = &
(1-t)^r \sum_{i=0}^\mu {\mu \choose i}  \frac{1}{i+r}t^{\mu - i} (1-t)^{i} ,
\end{eqnarray*}
it follows that
\[
\int_0^1 x^\mu P_n(x)dx = \frac{\lambda_n}r \int_0^1 Q_{n-r}(t)(1-t)^r \left[   \sum_{i=0}^\mu {\mu \choose i}  \frac{r}{i+r}t^{\mu - i} (1-t)^{i} \right] dt.
\]

Combining this with \ineq{pnn}  we have
\begin{eqnarray}\label{min} \nonumber
\lefteqn{ 1- \int_0^1 x^\mu P_n(x)dx} \\ \nonumber
&= & \int_0^1P_n(x)dx-\int_0^1 x^\mu P_n(x)dx\\ \nonumber
& = &
\frac{\lambda_n}r \int_0^1 Q_{n-r}(t)(1-t)^r \left[1-  \sum_{i=0}^\mu {\mu \choose i}  \frac{r}{i+r}t^{\mu - i} (1-t)^{i} \right] dt \\ \nonumber
& = &
\frac{\lambda_n}r \int_0^1 Q_{n-r}(t)(1-t)^r \left[ \sum_{i=0}^\mu {\mu \choose i}  \frac{i}{i+r}t^{\mu - i} (1-t)^{i} \right] dt \\ \nonumber
& =  &
\frac{\lambda_n}r \int_0^1 Q_{n-r}(t)(1-t)^{r+1}  \left[ \sum_{i=1}^\mu {\mu \choose i}  \frac{i}{i+r}t^{\mu - i} (1-t)^{i-1} \right] dt \\
&\le& c \lambda_n \int_0^1Q_{n-r}(t)(1-t)^{r+1}dt.
\end{eqnarray}

We now let
\[
m:= \left\lceil\frac n{8r}\right\rceil \andd Q_{n-r}(x):=\tau_m^{4r}(x).
\]
Then, $Q_{n-r}$ is a nonnegative polynomial and its  degree $\leq  4r(m-1)  \le n-r$.
Using \ineq{kls} and \ineq{kls1} we   have
\begin{eqnarray*}
  1-\int_0^1 x^\mu P_n(x)dx
&\le & c\lambda_n \left(\int_0^{x_1}+\int_{I_1}\right)Q_{n-r}(t)(1-t)^{r+1}dt\\
&\le &
c \lambda_n  |I_1|^{4r}\int_{-\infty}^{x_1}(1-t)^{1-3r}dt+  c \lambda_n   \int_{I_1}(1-t)^{r+1}dt\\
&\leq  & c \lambda_n|I_1|^{r+2} .
\end{eqnarray*}
Finally, recalling \ineq{lambda} we write
\begin{eqnarray*}
 1-\int_0^1x^\mu P_n(x)dx  & \leq & c  |I_1|^{r+2} \left( \int_0^1(1-t)^rQ_{n-r}(t)dt \right)^{-1} \\
&\le & c     |I_1|^{r+2} \left( \int_{I_1}(1-t)^rQ_{n-r}(t)dt \right)^{-1} \\
&\le & c   |I_1|^{r+2} \left( \int_{I_1}(1-t)^rdt \right)^{-1} \\
& \leq & c  |I_1|
 \leq  \frac c{n^2},
\end{eqnarray*}
and the proof of \ineq{small} is complete.
\end{proof}

\subsection{Proof of \thm{mainth}}\label{proofmain}

 Suppose   that $q\in\N_0$   and let
 $\gP_{n-2} := P_{n-2}$ where $P_n$ is the polynomial from the statement of \lem{lem210} with $r:= \max\{q-1, 1\}$. In particular, \ineq{small} with $\mu:=2$ implies
 that
 \[
 1 - \int_0^1 x^2 \gP_{n-2} (x) dx \leq {c_1 \over n^2}, \quad n\geq 3,
 \]
with the constant $c_1$ depending only on $q$.
Also, let $n_0 :=   2 \left\lceil c_1^{1/2}  \right\rceil \in \N$.

 For   $1\leq n < n_0$, we can define $M_n(f, x) := (1-x)f(0)+x f(1)$. Clearly, $M_n: C[0,1]\mapsto \Poly_1 \subset \Poly_n$ is a positive linear polynomial operator preserving linear functions as well as $k$-monotonicity for all $k$. Since $M_n(e_2, x) = x = x^2 + \varphi^2(x)$, \cor{maincor} (with $\a_n=1$) implies that
\[
|f(x)-M_n(f, x)| \leq c   \left(  1 +   \frac{ \varphi^{2-\lambda}(x)}{h^2 (\varphi(x)+1/4)^\lambda} \right) \omega_2^{\varphi^\lambda} (f, h),
\]
and the statement of \thm{mainth} follows.

 Suppose now that
  $n\geq n_0$ is fixed, and define $M_n(f, \cdot) := H_{n}(\gP_{n-2}; f, \cdot)$.
It follows from Lemmas~\ref{mainspa} and \ref{lem210} that $M_n: C[0,1]\mapsto\Poly_n$ is a positive linear operator preserving linear functions as well as $k$-monotonicity for all $0\leq k\leq q$, and $M_n(e_2, x) - x^2 = \a_n \varphi^2(x)$ with
\[
\a_n = 1 - \int_0^1 t^2 \gP_{n-2}(t) dt   \leq   \frac{c_1}{n^2} \leq  \frac{1}{4} .
\]
 Therefore, taking into account that
 the function $\Lambda(t):= t\left( \varphi(x)+ \sqrt{t}\right)^{-\lambda}$ is increasing for $t\in [0, \infty)$ if $0\leq \lambda < 2$,
 \cor{maincor} yields, for  $0<h\leq c_0$,
\begin{eqnarray*}
|f(x)-M_n(f, x)| &\leq&
  c  \left(1+   \frac{ \a_n \varphi^{2-\lambda}(x) }{h^2  \left(\varphi(x)+  \sqrt{\a_n}\right)^\lambda    } \right) \omega_2^{\varphi^\lambda} (f, h) \\
 &\leq&
 c  \left(1+   \frac{ c_1 \varphi^{2-\lambda}(x) }{h^2 n^2 \left(\varphi(x)+  \sqrt{c_1/n^2 }\right)^\lambda    } \right) \omega_2^{\varphi^\lambda} (f, h),
\end{eqnarray*}
which implies \ineq{maines}.

\sect{Appendix: Bernstein-Durrmeyer-Lupa\c{s} polynomials with ultraspherical weights}\label{append}

The main results in this paper (as well as all results from \cite{cggkz} and \cite{g} that we need) greatly depend on (in our opinion, a rather interesting) paper by A. Lupa\c{s} \cite{lup} which does not seem to be readily available. Hence, in this section, we state and provide alternative elementary proofs for all theorems from \cite{lup} that we use.

For $\alpha >-1$, let
\be \label{defphi}
\phi_n^{(\alpha)}(x)  := {(-1)^n \over (\alpha+1)_n} x^{-\alpha} (1-x)^{-\alpha} {d^n \over dx^n} \left\{ x^{n+\alpha} (1-x)^{n+\alpha} \right\}
\ee
be the (shifted) ultraspherical polynomials on $[0,1]$ (normalized so that $\phi_n^{(\alpha)}(1)=1$). Note that
\[
\phi_n^{(\alpha)}(x) = {P_n^{(\a+1/2)}(2x-1) \over P_n^{(\a+1/2)}(1)} ,
\]
where $P_n^{(\lambda)}$
is the classical ultraspherical (Gegenbauer)  polynomial  (see \cite{szego}*{Chapter IV}).
 Recall that
\[
P_n^{(\lambda)}(1) = {\binom{n+2\lambda-1}{n}} = {(2\lambda)_n \over n!} .
\]

\begin{remark}
With  $\phi_0^{(\alpha)}(x)=1$ and $\phi_1^{(\alpha)}(x)=2x-1$,
the following recurrence equation is valid (see \cite{szego}*{(4.7.17)}:
\be \label{recur}
(n+2\a) \phi_n^{(\alpha)}(x) = (2n+2\a-1)(2x-1) \phi_{n-1}^{(\alpha)}(x) - (n-1) \phi_{n-2}^{(\alpha)}(x) , \quad n\geq 2.
\ee
In particular, this implies that, if $\phi_n^{(\alpha)}(x) = \lambda_n^{(\alpha)} x^n + p_{n-1}(x)$ with $p_{n-1}\in \Poly_{n-1}$, then
\be \label{highest}
\lambda_n^{(\alpha)} := {4^n (\a+1/2)_n \over (2\a+1)_n } = {(2\a+n+1)_n \over (\a+1)_n}
\ee
(see also \cite{szego}*{(4.7.9)}).
\end{remark}

Bernstein-Durrmeyer-Lupa\c{s} polynomials with ultraspherical weights are defined as
\be \label{dalpha}
\dal_n (f, x) := \sum_{k=0}^n p_{n,k}(x) \frac{\langle p_{n,k}, f \rangle}{\langle p_{n,k}, 1 \rangle} ,
\ee
 where
 \[
 \langle f, g \rangle :=   \int_0^1 f(t) g(t) d w(t, \alpha) , \quad d w(t, \alpha) := \frac{t^\alpha (1-t)^\alpha}{B(\alpha+1, \alpha+1)} dt ,
 \]
and $B(x,y) := \int_0^1 t^{x-1} (1-t)^{y-1} dt$ is the beta function.
Note that
\begin{eqnarray*}
\langle p_{n,k}, 1 \rangle &=& \frac{1}{B(\alpha+1, \alpha+1)} {\binom nk} \int_0^1   t^{\alpha+k} (1-t)^{\alpha+n-k}  dt \\
&=&  \frac{B(\alpha+k+1, \alpha+n-k+1)}{B(\alpha+1, \alpha+1)}{\binom nk} \\
& = & {\binom nk} \frac{(\alpha+1)_k (\alpha+1)_{n-k}}{(2\alpha+2)_n} ,
\end{eqnarray*}
where  we used the fact that  $B(x,y)=\Gamma(x)\Gamma(y)/\Gamma(x+y)$, where $\Gamma(x) := \int_0^\infty t^{x-1} e^{-t} dt$ is the gamma function, and $\Gamma(x+1)=x\Gamma(x)$, $x>0$.

\begin{remark} \label{rem32}
If $\a=0$, then
$D^{\langle 0 \rangle}(f,x) =  D_n (f, x)$, where $D_n$
is the (usual) Bernstein-Durrmeyer operator defined in \ineq{bd}.
\end{remark}

\begin{lemma}[\cite{lup}*{(1.3) and (3.2)}]
For any $\alpha >-1$,
\be \label{phiber}
\phi_n^{(\alpha)}(x) =  (\alpha+1)_n \sum_{k=0}^n {(-1)^{n-k}  \over (\alpha+1)_k (\alpha+1)_{n-k}}      p_{n,k} \left(x\right) .
\ee
and, for $t\neq 1-x$,
\be \label{lup2}
(x+t-1)^n \phi_n^{(\alpha)}\left( {xt \over x+t-1} \right) =
(\a+1)_n  \sum_{k=0}^n {p_{n,k}(x) p_{n,k}(t) \over {\binom nk} (\alpha+1)_k (\alpha+1)_{n-k}} .
\ee
\end{lemma}

Note that  \ineq{lup2} corrects a misprint in \cite{lup}*{(3.2)}. Also, we remark that taking the limit in \ineq{lup2} as $t \to 1-x$ we get the identity
\[
\lambda_n^{(\alpha)} = (\a+1)_n  \sum_{k=0}^n {\binom nk} {1 \over  (\alpha+1)_k (\alpha+1)_{n-k}} .
\]

\begin{proof} First of all,
\begin{eqnarray*}
\lefteqn{ {d^n \over dx^n} \left\{ x^{n+\alpha} (1-x)^{n+\alpha} \right\} }\\
& = &
 \sum_{k=0}^n {\binom nk} {d^{n-k} \over dx^{n-k}} x^{n+\alpha} {d^{k} \over dx^{k}} (1-x)^{n+\alpha} \\
 & = &
 \sum_{k=0}^n {\binom nk} {(\alpha+1)_n \over (\alpha+1)_k} x^{\alpha+k} {(\alpha+1)_n \over (\alpha+1)_{n-k}} (-1)^{k}  (1-x)^{n+\alpha-k} \\
  & = &
  [(\alpha+1)_n]^2 x^{\alpha}(1-x)^{ \alpha} \sum_{k=0}^n   {(-1)^{k} \over (\alpha+1)_k (\alpha+1)_{n-k}}   p_{n,k} \left(x\right) ,
 \end{eqnarray*}
 which together with \ineq{defphi} implies \ineq{phiber}.

 Now, since
\[
p_{n,k}(x) p_{n,k}(t) = (-1)^{n-k} {\binom nk}  (x+t-1)^n p_{n,k} \left( {xt \over x+t-1} \right) ,
\]
using \ineq{phiber} we have
\begin{eqnarray*}
\lefteqn{  \sum_{k=0}^n {p_{n,k}(x) p_{n,k}(t) \over {\binom nk} (\alpha+1)_k (\alpha+1)_{n-k}} }\\
& = &
(x+t-1)^n
\sum_{k=0}^n {(-1)^{n-k} \over  (\alpha+1)_k (\alpha+1)_{n-k}}  p_{n,k} \left( {xt \over x+t-1} \right)    \\
& = &
{1 \over (\a+1)_n } (x+t-1)^n \phi_n^{(\alpha)}\left( {xt \over x+t-1} \right) ,
\end{eqnarray*}
which is \ineq{lup2}.
\end{proof}

\begin{theorem}[\cite{lup}*{Theorem 4.1}] \label{th33}
For any $\a>-1/2$, $f\in  C[0,1]$ and $x\in [0,1]$,
\[
\dal_n(f, x) = \frac{(2\alpha+2)_n}{(\alpha+1)_n}
\int_0^1 f(t) \int_0^1 \left[\Theta (x,t,u)\right]^n   \, d w(u,\alpha-1/2)\, dw(t,\alpha) ,
\]
where $\Theta (x,t,u) := (1-u) a(x,t)+ u b(x,t)$ with
\[
 a(x,t) := \left( \sqrt{xt}-\sqrt{(1-x)(1-t)} \right)^2 \andd
 b(x,t) := \left( \sqrt{xt}+\sqrt{(1-x)(1-t)} \right)^2 .
\]
\end{theorem}

\begin{proof}  Using the definition \ineq{dalpha} we have, for any $\a>-1$,
\begin{eqnarray} \label{kernel}\nonumber
\lefteqn{ \dal_n (f, x)  =   \int_0^1 f(t) \left[ \sum_{k=0}^n p_{n,k}(x) p_{n,k}(t)/\langle p_{n,k}, 1 \rangle \right]    d w(t, \alpha) } \\ \nonumber
& = &
(2\a+2)_n \int_0^1 f(t) \left[
\sum_{k=0}^n {p_{n,k}(x) p_{n,k}(t) \over {\binom nk} (\alpha+1)_k (\alpha+1)_{n-k}} \right]    d w(t, \alpha) \\
& = &
{ (2\a+2)_n \over (\a+1)_n}
\int_0^1 f(t) \left[
(x+t-1)^n \phi_n^{(\alpha)}\left( {xt \over x+t-1} \right)  \right]    d w(t, \alpha) ,
\end{eqnarray}
and it remains to prove that, for $\a>-1/2$,
\be \label{aux}
(x+t-1)^n \phi_n^{(\alpha)}\left( {xt \over x+t-1} \right) = \int_0^1 \left[\Theta (x,t,u)\right]^n   \, d w(u,\alpha-1/2) .
\ee
This identity immediately follows from
  Gegenbauer's formula (see, \eg \cite{ss}*{(2)} or \cite{szego}*{(4.10.3)}): for $\lambda >0$ and all real $x$,
\[
{P_n^{(\lambda)}(x) \over P_n^{(\lambda)}(1)} = {\Gamma(\lambda+1/2) \over \sqrt{\pi}\, \Gamma(\lambda)} \int_0^\pi \left[ x +\sqrt{x^2-1} \cos t \right]^n \sin^{2\lambda-1} t\, dt ,
\]
recalling that $\phi_n^{(\alpha)}(x) =  P_n^{(\a+1/2)}(2x-1)/P_n^{(\a+1/2)}(1)$ and changing variables $\cos t = 2u-1$.
Alternatively, \ineq{aux} can be proved by induction using the recurrence equation \ineq{recur}. Yet another way to prove \ineq{aux} is to use  several results
from the theory of hypergeometric functions as was originally done by \lupas{} in \cite{lup}.
\end{proof}

Since  $\Theta (x,t,u) \in [0,1]$, for all $x,t,u \in [0,1]$, one can immediately get a result on positive summation of a sequence of operators $\dal_n$ as a corollary of \thm{th33} (see \cite{lup}*{Theorem 5.2(2)}). We state this corollary in a slightly more general form which is useful for applications.

\begin{corollary} \label{keycor3}
 Let $\a>-1/2$ and $n, r, \varrho\in \N_0$ with $0\leq \varrho\leq r \leq n$,  and let a  (generating) polynomial
$\gP_n(x) = \sum_{k=0}^n a_k x^k$ be such that
\[
\gP_n^{(r)}(x) \geq 0,  \quad \mbox{\rm for all }\; x \in [0,1] .
\]
 Then, for any   function $\sigma$ such that $0\leq \sigma(x)\leq 1$, $x\in [a,b]\subset [0,1]$,
 \begin{eqnarray}  \label{in:positive3}\nonumber
 \lefteqn{ L_n^{\langle \alpha \rangle} (\gP_n, \sigma(\cdot), r, \varrho, [a,b] \, ; f,x)} \\
 &:=&
 \sum_{k=r-\varrho}^{n-\varrho} {(\a+1)_{k} (k-r+\varrho+1)_r \over (2\a+2)_{k}}  \,a_{k+\varrho }
 \left[ \sigma(x)\right]^k D_k^{\langle \alpha \rangle} (f,x)
 \end{eqnarray}
is a {\bf positive} linear operator on $C[a,b]$.

In particular, if $r=\varrho=0$ and $\sigma(x)=1$, $x\in [0,1]$, then
 \[
 L_n^{\langle \alpha \rangle} (\gP_n, 1, 0, 0, [0,1]; f,x)  =
 \sum_{k=0}^n {(\a+1)_k \over (2\a+2)_k} \, a_k D_k^{\langle \alpha \rangle} (f,x)
 \]
is a {\bf positive} linear operator.
\end{corollary}

\begin{proof}
 Since
\[
\Q_{n-\varrho}(x) := x^{r-\varrho} \gP_n^{(r)}(x)  = \sum_{k=r}^n (k-r+1)_r \, a_k \, x^{k-\varrho}
=
\sum_{k=r-\varrho}^{n-\varrho} (k-r+\varrho+1)_r \, a_{k+\varrho } \, x^{k}
\]
we have
 \begin{eqnarray*}
\lefteqn{ L_n^{\langle \alpha \rangle} (\gP_n, \sigma, r, \varrho\, ; f,x) }\\
 & =&
 \sum_{k=r-\varrho}^{n-\varrho} (k-r+\varrho+1)_r   \,a_{k+\varrho }
 \left[ \sigma(x)\right]^k \int_0^1 f(t) \int_0^1 \left[\Theta (x,t,u)\right]^k   \, d w(u,\alpha-1/2)\, dw(t,\alpha) \\
 & = &
 \int_0^1 f(t)   \int_0^1
 \sum_{k=r-\varrho}^{n-\varrho} (k-r+\varrho+1)_r   \,a_{k+\varrho } \left[\sigma(x) \Theta (x,t,u)\right]^k
 \, d w(u,\alpha-1/2)\, dw(t,\alpha) \\
 & = &
 \int_0^1 f(t)   \int_0^1  \Q_{n-\varrho} \left[\sigma(x) \Theta (x,t,u) \right]  \, d w(u,\alpha-1/2)\, dw(t,\alpha) .
\end{eqnarray*}
In view of the fact that $0\leq \sigma(x) \Theta (x,t,u)\leq 1$, for all $x,t,u \in [a,b]$, and that  $ \Q_{n-\varrho}$ is nonnegative on $[0,1]$, we conclude that the operator $L_n^{\langle \alpha \rangle}$ is positive.
\end{proof}

\begin{lemma}[\cite{lup}*{Lemma 4.2}] \label{lem:deriv}
For $\a>-1$, $n,\nu\in\N$ and $f\in C^{\nu}[0,1]$,
\be \label{deriv}
\frac{d^\nu}{dx^\nu} \dal_n(f, x) = {n! \over (n-\nu)! (n+2\a+2)_\nu} D^{\langle \alpha+\nu \rangle}_{n-\nu}\left( f^{(\nu)}, x \right) .
\ee
\end{lemma}

\begin{proof} It is sufficient to prove \ineq{deriv} for $\nu=1$ since the general case immediately follows by induction.

It follows from \ineq{kernel}   that, for $\a>-1$,
\[
  \dal_n (f, x)  =  { (2\a+2)_n \over (\a+1)_n}  \int_0^1 f(t) K_n^{\langle \alpha  \rangle}(x,t)    d w(t, \alpha)    ,
\]
where
\[
K_n^{\langle \alpha  \rangle}(x,t) := (x+t-1)^n \phi_n^{(\alpha)}\left( {xt \over x+t-1} \right) ,
\]
and \ineq{deriv} with $\nu=1$ follows using integration by parts and the following identity:
\be \label{byparts}
 \frac{\partial}{\partial x}K_n^{\langle \alpha  \rangle}(x,t) =    n(2t-1) K_{n-1}^{\langle \alpha+1  \rangle}(x,t) -  {n t(1-t)\over \a+1}   \frac{\partial}{\partial t}K_{n-1}^{\langle \alpha+1  \rangle}(x,t) .
\ee
Using
\be \label{fder}
\frac{d}{dz} \phi_n^{(\alpha)}(z)  =  {n(2\a+n+1)  \over \a+1} \phi_{n-1}^{(\alpha+1)}(z)
\ee
 (see, \eg \cite{szego}*{(4.7.14)}) identity \ineq{byparts} can be rewritten as
\be \label{newby}
\phi_n^{(\alpha)}(z) = (2z-1) \phi_{n-1}^{(\alpha+1)}(z) - {(n-1)(2\a+n+2) \over (\a+1) (\a+2)} z(1-z) \phi_{n-2}^{(\alpha+2)}(z) .
\ee
Finally, \ineq{newby} can be proved using the  ``reduction of $\alpha$'' formula
\[
 z(1-z) \phi_{n-1}^{(\alpha+1)}(z) ={\a+1 \over 2n} \left( (2z-1) \phi_{n}^{(\alpha)}(z) - \phi_{n+1}^{(\alpha)}(z)\right)
\]
(see, \eg \cite{szego}*{(4.7.27)}) and the recurrence equation \ineq{recur}.
Alternatively, one can use the formula for the $\nu$th derivative of $\phi_n^{(\alpha)}$ that follows from \ineq{fder}
\[
\frac{d^\nu}{dz^\nu} \phi_n^{(\alpha)}(z)  =  {(n-\nu+1)_\nu (2\a+n+1)_\nu  \over (\a+1)_\nu} \phi_{n-\nu}^{(\alpha+\nu)}(z) , \quad 1\leq \nu \leq n ,
\]
and the fact that both sides of \ineq{newby} are polynomials of degree $n$ whose $\nu$th derivatives are the same at $z=1$ for all $0\leq \nu\leq n$.
\end{proof}

\lem{lem:deriv} can be used to recursively calculate $\dal_n(e_i, x)$, $i\in\N_0$, taking into account that
\[
\dal_n(e_i, 0) = \frac{\langle p_{n,0}, e_i \rangle}{\langle p_{n,0}, 1 \rangle}
= {B(\a+i+1, \a+n+1) \over B(\a+1, \a+n +1)} = {(\a+1)_i \over (n+2\a+2)_i }.
\]
For example,
\[
\dal_n(e_0, x)=1, \quad \dal_n(e_1, x) = {nx +\a+1 \over n+2\a +2}
\]
and
\[
\dal_n(e_2, x)=
{n(n-1)x^2 + 2n(\a+2)x + (\a+1)(\a+2) \over (n+2\a+2)(n+2\a+3)} .
\]

\end{document}